\documentclass[reqno]{amsart}
\usepackage[scale=0.75, centering, headheight=14pt]{geometry}
\usepackage[latin1]{inputenc}
\usepackage[T1]{fontenc}
\usepackage{lmodern}
\usepackage[english]{babel}
\usepackage{stmaryrd}
\usepackage{tikz}
\usepackage{bbm}
\usepackage{amsmath,amssymb,amsfonts,amsthm}
\usepackage{mathtools,accents}
\usepackage{mathrsfs}
\usepackage{xfrac}
\usepackage{array} 
\usepackage{aliascnt}
\usepackage{booktabs} 
\usepackage{array} 

\usepackage{verbatim} 
\usepackage{subfig} 
\usepackage{mathrsfs}
\usepackage{mathrsfs, dsfont}
\usepackage{amssymb}
\usepackage{amsthm}
\usepackage{amsmath,amsfonts,amssymb,esint}
\usepackage{graphics,color}
\usepackage{enumerate}
\usepackage{mathtools,centernot}

\usepackage{microtype}
\usepackage{paralist} 
\usepackage{cases}
\usepackage[initials,alphabetic]{amsrefs}
\allowdisplaybreaks

\usepackage{braket}
\usepackage{bm}

\usepackage[citecolor=blue,colorlinks]{hyperref}
\addto\extrasenglish{}

\usepackage{enumerate}
\usepackage{xcolor}

\usepackage{aliascnt}

\makeatletter
\def\newaliasedtheorem#1[#2]#3{
  \newaliascnt{#1@alt}{#2}
  \newtheorem{#1}[#1@alt]{#3}
  \expandafter\newcommand\csname #1@altname\endcsname{#3}
}
\makeatother

\theoremstyle{plain}

\newaliasedtheorem{lemma}[theorem]{Lemma}
\newaliasedtheorem{prop}[theorem]{Proposition}
\newaliasedtheorem{claim}[theorem]{Claim}
\newaliasedtheorem{corollary}[theorem]{Corollary}

\newtheorem{maintheorem}{Theorem}

\DeclareMathOperator{\diver}{div}

\theoremstyle{definition}
\newaliasedtheorem{definition}[theorem]{Definition}
\newaliasedtheorem{example}[theorem]{Example}

\newtheorem{OQ}[]{Open Question}

\theoremstyle{remark}
\newaliasedtheorem{remark}[theorem]{Remark}
\newaliasedtheorem{ex}[theorem]{Example}

\numberwithin{equation}{section}

\def\R{\mathbb R}
\def\N{{\mathbb N}}
\def\Z{{\mathbb Z}}
\def\T{{\mathbb T}}

\DeclareMathOperator{\supp}{supp}

\DeclareMathOperator{\initial}{in}

\newcommand{\p}{\partial}

\newcommand{\vt}{\tilde{\vartheta}}
\newcommand{\cc}{\mbox{cs}}
\newcommand{\ds}{\mbox{ds}}

\title[Anomalous dissipation for Onsager's solutions]
 {Onsager critical solutions of the forced \\ Navier-Stokes equations}
\author[Elia Bru\`e, Maria Colombo, Gianluca Crippa, Camillo De Lellis \and Massimo Sorella]{Elia Bru\`e, Maria Colombo, Gianluca Crippa, Camillo De Lellis \and Massimo Sorella}

\address{Elia Bru\`e 
\hfill\break School of Mathematics, Institute for Advanced Study, 1 Einstein Dr., Princeton NJ 05840, U.S.A.}
\email{elia.brue@math.ias.edu}

\address{Maria Colombo  
\hfill\break EPFL B, Station 8, CH-1015 Lausanne, Switzerland}
\email{maria.colombo@epfl.ch}

\address{Gianluca Crippa
\hfill\break Departement Mathematik und Informatik, Universit\"at Basel, Spiegelgasse 1, CH-4051 Basel, Switzerland}
\email{gianluca.crippa@unibas.ch}

\address{Camillo De Lellis
\hfill\break School of Mathematics, Institute for Advanced Study, 1 Einstein Dr., Princeton NJ 05840, U.S.A.}
\email{camillo.delellis@math.ias.edu}

\address{Massimo Sorella
\hfill\break EPFL B, Station 8, CH-1015 Lausanne, Switzerland}
\email{massimo.sorella@epfl.ch}

\begin{document}

\begin{abstract}
We answer positively to \cite{BDL22}*{Question 2.4} by building new examples of solutions to the forced $3d$-Navier-Stokes equations with vanishing viscosity, which exhibit anomalous dissipation and which enjoy uniform bounds in the space $L^3_t C^{\sfrac{1}{3}-\varepsilon}_x$, for any fixed $\varepsilon >0$.
Our construction combines ideas of~\cite{BDL22} and~\cite{CCS22}.
\end{abstract}

\maketitle

\section{Introduction}

The forced Navier--Stokes equations  on the $3$-dimensional torus $\T^3 \simeq \R^3 / \Z^3$ are given by
\begin{align} 
\tag{NS}
\label{e:NSE}
\partial_t v_{\nu} + v_\nu \cdot \nabla v_\nu + \nabla p_\nu = \nu \Delta v_\nu + F_\nu
\\
\diver v_\nu =0, \notag
\end{align}
where $v_\nu : [0, T] \times \T^3 \to \R^3$ is the velocity field, $p_\nu : [0, T] \times \T^3 \to \R$ is the pressure, $\nu> 0$ is the viscosity parameter and $F_\nu : [0, T] \times \T^3 \to \R^3$ is a (divergence-free) force that may depend on $\nu$. When $\nu =0$ the Navier--Stokes equations~\eqref{e:NSE} reduce to the forced Euler equations
\begin{align} 
\tag{E}
\label{e:E}
\partial_t v_0 + v_0 \cdot \nabla v_0 + \nabla p_0 =  F_0
\\
\diver v_0 =0. \notag
\end{align}
We consider both the Navier--Stokes equations \eqref{e:NSE} and the Euler equations \eqref{e:E} with a prescribed initial datum $v_{\initial}$ which is independent of the  viscosity parameter $\nu$, namely
\begin{equation}\label{e:Cauchy}
v_\nu  (0, \cdot) = v_{\initial} \, .
\end{equation}
Following \cite{BDL22} we study {\em smooth} solutions of \eqref{e:NSE} (namely $u_\nu$ and $F_\nu$ are both $C^\infty$), which enjoy uniform in $\nu$ bounds for $v_\nu$ and $F_\nu$ in appropriate function spaces $X$ and $Y$. The purpose is to understand which spaces $X$ and $Y$ allow for $u_\nu$ to display {\em anomalous dissipation}, more precisely whether 
\begin{equation}\label{diss_main_Onsager} 
\limsup_{\nu \downarrow 0}  \, \nu \int_0^T \int_{\T^3} | \nabla v_\nu|^2 \, dx\,dt > 0 \, .
\end{equation}
We require that the space $Y$ rules out anomalous dissipation for solutions of the forced linear Stokes equations under the assumption $\sup_\nu \|F_\nu\|_Y < \infty$, namely \eqref{diss_main_Onsager} would not hold if we eliminate the nonlinear advective term $v_\nu \cdot \nabla v_\nu$ from \eqref{e:NSE} and we have uniform bounds for the body forces in the space $Y$. 
As it is noticed in \cite{BDL22}*{Section 2} the assumption
\begin{equation}\label{bound_force}
\sup_\nu \| F_\nu \|_{L^{1+ \sigma} ([0,1]; C^{\sigma} (\T^3))} <\infty
\end{equation}
for any positive $\sigma>0$ is in fact sufficient. 

In \cite{BDL22} the first and fourth authors give examples of smooth solutions $v_\nu$ to \eqref{e:NSE} for which:
\begin{itemize}
\item[(i)]  \eqref{bound_force} holds (in fact with the stronger bound $\sup_\nu \|F_\nu\|_{L^\infty_t (C^{1-\varepsilon})}<\infty$ for any given positive $\varepsilon$),
\item[(ii)] $\sup_\nu \|v_\nu\|_{L^\infty} < \infty$,
\item[(iii)] and \eqref{diss_main_Onsager} is satisfied.
\end{itemize}
In \cite{BDL22}*{Section 2} the authors ask whether this type of behavior is still possible if  the uniform $L^\infty$ bound~(ii) is replaced by a uniform bound in some space $X$ which is close to be ``Onsager critical''. The Onsager criticality refers to the famous remark by Onsager \cite{onsager} that if $\|v\|_{L^\infty (C^{\sfrac{1}{3}+\varepsilon})} < \infty$ and $u$ solves~\eqref{e:E} with $F=0$, then such solution $u$ is energy conservative. After a first partial result by Eyink in \cite{eyinkonsag}, the latter was rigorously proved by Constantin, E, and Titi in \cite{ConstantinETiti}. It is straightforward to check that, using the arguments in \cite{ConstantinETiti}, \eqref{bound_force} and a uniform bound in $\|v_\nu\|_{L^3 (C^{\sfrac{1}{3}+\varepsilon})}$ is in fact enough to rule out~\eqref{diss_main_Onsager}. 

Onsager in \cite{onsager} stated also that the regularity class $L^\infty_t (C^{\sfrac{1}{3}}_x)$ should in fact be critical, in particular he conjectured the existence of solutions of \eqref{e:E} with $F=0$ belonging to slightly lower regularity classes of $L^\infty_t (C^{\sfrac{1}{3}}_x)$ which do not conserve the kinetic energy. After a decade of work in the area which started with~\cites{DLS1,DLS2}, the Onsager conjecture was proved by Isett in \cite{Isett} (cf. also \cite{BDSV}) using ``convex intgeration methods''. 

While Onsager's conjecture was motivated by the zero-th law of Kolmogorov's fully developed turbulence, which roughly speaking states that \eqref{diss_main_Onsager} should be a ``typical'' phenomenon, 
it seems at the moment very hard to show that at least some of the dissipative solutions of the unforced Euler equations found so far in the literature can actually be approximated by a sequence of regular solutions to the unforced Navier-Stokes. For this reason in \cite{BDL22} the authors suggested to consider the forced versions of both equations. The main result of this paper is to show that indeed \eqref{diss_main_Onsager} can be achieved for family of solutions $\{v_\nu\}_{\nu}$ which enjoy a uniform bound in a space which is just below the Onsager-critical $L^3_t (C^{\sfrac{1}{3}}_x)$, while the corresponding forces $F_\nu$ also enjoy a bound like \eqref{bound_force} which rules out \eqref{diss_main_Onsager} for solutions of the linear Stokes equations.

\begin{maintheorem}[Anomalous dissipation]\label{t_Onsager} 
Let $T=1$. For any $\alpha < 1/3$ there exist $\sigma>0$, a divergence-free initial datum $v_{in} \in C^\infty(\T^3; \R^3) $ with $ \int_{\T^3} v_{\initial}=0$, and a family of forces $\{ F_\nu \}_{\nu > 0} \subset C^\infty ( [0,1] \times \T^3; \R^3)$ satisfying \eqref{bound_force} such that:
\begin{itemize}
\item for each $\nu>0$ there is a unique solution to \eqref{e:NSE} with $v^\nu(0,\cdot)=v_{\rm in}(\cdot)$ which satisfies
\begin{align}\label{bound_main_Onsager} 
\sup_{\nu \in [0,1]} \big(\| v_\nu \|_{L^{3} ([0,1]; C^\alpha(\T^3))}  +  \| v_\nu \|_{L^{\infty } ([0,1] \times \T^3)}\big)  < \infty \, ,
\end{align}
\item \eqref{diss_main_Onsager} holds.
\end{itemize}
Furthermore, we have that $F_{\nu} \to F_0$ in $L^{1+ \sigma}((0,1); C^\sigma (\T^3))$ and $v_\nu \to v_0$ in $L^2((0,1) \times \T^3 )$ as $\nu \to 0$, and in particular $(v_0, p_0, F_0)$ is a solution of \eqref{e:E}.
\end{maintheorem}

\begin{remark}
In our construction all the dissipation occurs at the time $T=1$, namely \eqref{diss_main_Onsager} fails at any~$T<1$.
In a forthcoming paper \cite{DRI22} De Rosa and Isett point out that this type of ``instantaneous loss of energy'' cannot occur at a time $T \in (0,1)$ for solutions belonging to $L^p([0,1];C^{\sfrac{1}{3}-}(\T^3))$ for any $p>3$. 
\end{remark}

\begin{remark}\label{r:remark-weakly}
If we only required that the forces $\{ F_\nu \}_{\nu > 0}$ were uniformly bounded in $L^1 ((0,1); L^\infty (\T^3))$, then anomalous dissipation would be already possible for solutions of the forced heat equation. Indeed, for any $\nu \in (0,1)$ such that $\nu^{-1/2} \in \N$ we can consider $\vartheta_\nu : [0,1] \times \T^3 \to \R$ defined as
$$
\vartheta_{\nu} (t,x) = (e^{- 4 \pi^2 t} -1 ) \sin (2 \pi \nu^{-1/2} x),
$$
and observe that it solves 
\begin{align}\label{e:heat}
\begin{cases}
\partial_t \vartheta_\nu - \nu \Delta \vartheta_\nu = - 4 \pi^2 \sin (2 \pi \nu^{-1/2} x) =: F_\nu 
\\
\vartheta_{\nu}(0, \cdot) \equiv 0.
\end{cases}
\end{align} 
It is straightforward to check that $\nu \int_0^1 \int_{\T^3} |\nabla  \vartheta_\nu (t,x)|^2 dx dt \geq 1/4$ for every $\nu \in (0,1)$ as above. The latter example can be easily modified to produce an analogous one for the linear Stokes equations.

 Note that the crucial point is in the oscillations introduced by the sequence $F_\nu$. In particular, strong convergence in $L^1_t L^2_x$ of $F_\nu$ would actually suffice to show that the unique solutions of \eqref{e:heat} satisfy $\nu \int_0^1 \int_{\T^3} |\nabla  \vartheta_\nu (t,x)|^2 dx dt  \to 0$.
\end{remark}

The following open question was also raised in \cite{BDL22} and at present the methods of this work do not seem strong enough to address it. 

\begin{OQ}
Can Theorem \ref{t_Onsager} be shown for Leray solutions but replacing $F_\nu$ with a $\nu$-independent force in the space $L^1( (0,2) ; L^\infty (\T^3))$?
\end{OQ}

 In view of Remark \ref{r:remark-weakly} even producing one such example with force in $L^1 ((0,2); L^2 (\T^3))$ seems interesting and highly nontrivial.

\subsection{Lack of selection principle and non-uniqueness}

As in \cite{CCS22}, a byproduct of our techniques is the lack of a selection principle under vanishing viscosity for bounded solutions of the three dimensional forced Euler equations, if the force converges in the vanishing viscosity limit. We say that a weak solution $v \in L^\infty ((0,T); L^2(\T^3))$ of the forced Euler equations \eqref{e:E} is admissible if  
\begin{align} \label{admissible}
    \int_{\T^3} |v(x,t)|^2 dx \leq \int_{\T^3} |v_{\initial}(x)|^2 dx + 2 \int_{\T^3} F (x,t) \cdot v(x,t) dx 
\end{align} 
for $a.e.$ $t\in (0,T)$.

We will show that the problem of uniqueness and vanishing viscosity selection in the class of admissible solutions for \eqref{e:E} is related to having a solution in the space $L^1_t (W^{1,\infty}_x)$  (this is essentially the threshold for classical ``weak-strong'' uniqueness results, see e.g. \cites{W18,DRIS22}). In particular uniqueness and selection both fail for solutions in $L^1 ((0,T); C^\alpha (\T^3))$ for any $\alpha <1$. 

\begin{remark}
The nonuniqueness of admissible solutions has been already shown in the class $C^{\beta} ((0,T) \times \T^3)$ for $\beta < 1/3$ for the unforced Euler equations using the convex integration technique, cf.~the aforementioned papers \cites{DLS1,DLS2,Isett,BDSV}. 
\end{remark}

\begin{maintheorem}[Nonuniqueness and lack of selection I]\label{t_selection2}
Let $T=2$ and let $\alpha ' \in [0, 1)$ be given. Then there are:
\begin{itemize}
\item[(a)] $\sigma>0$ and  a family of smooth body forces $F_\nu$ satisfying \eqref{bound_force},
\item[(b)] a limit $F_0$ such that $ F_\nu \to F_0 $  in $L^{1+ \sigma} ((0,2) ; C^\sigma ( \T^3))$,
\item[(c)] a divergence-free initial datum $v_{\initial} \in C^\infty(\T^3)$ with $ \int_{\T^2} v_{\initial}=0$, 
\item[(d)] and a family $\{ v_\nu \}_{\nu >0}$ of (unique) smooth solutions of \eqref{e:NSE} and \eqref{e:Cauchy}
\end{itemize}
such that the following holds:
\begin{itemize}
\item[(i)] $\sup_{\nu \in [0,1]}   \| v_\nu \|_{L^{\infty } ((0,2) \times \T^3)} \leq 1$;
\item[(ii)] $\{v_\nu\}_{\nu >0}$ has at least two distinct limit points, as $\nu \to 0$, in the $L^\infty$ weak$^*$ topology, which are two distinct bounded admissible solutions $v_0^{\cc}$ and $v_0^{\ds}$ of \eqref{e:E} and \eqref{e:Cauchy};
\item[(iii)] furthermore, $v_0^{\cc} \in L^{1}((0,2); C^{\alpha ' } (\T^3)) \cap L^\infty$ satisfies the following energy balance
\begin{align} \label{eq:conservative}
 \| v_0^{\cc} (t, \cdot ) \|_{L^2}^2 =  \| v_{\initial} \|_{L^2}^2 + 2 \int_0^t \int_{\T^3} F_0 \cdot v_0^{\cc}   
\qquad \text{for a.e. $t \in (0,2)$,}
\end{align}
while $v_0^{\ds} \in L^\infty$ exhibits the strict dissipation 
\begin{align} \label{eq:dissipative}
\| v_0^{\ds} (t, \cdot ) \|_{L^2}^2  <  \frac{\| v_{\initial} \|_{L^2}^2}{2} + 2 \int_0^t \int_{\T^3} F_0 \cdot v_0^{\ds} \, 
\qquad \text{for any $t \in [1,2)$.}
\end{align} 
\end{itemize}
\end{maintheorem}

If we give up the regularity of the conservative solution $v_0^{\cc}$ it is possible to show nonuniqueness and lack of selection for much smoother forces.

\begin{maintheorem}[Nonuniqueness and lack of selection II] \label{t_selection}
Let $T=2$ and let $\alpha ' \in [0,1)$ be given. Then there are:
\begin{itemize}
\item[(a)] a family $\{F_\nu\}_{\nu >0}$ of smooth forces and a limiting $F_0$ such that 
$ F_\nu \to F_0 $  in $C^{\alpha '} ((0,2) \times \T^3)$,
\item[(b)] a divergence-free initial datum $v_{\initial} \in C^\infty(\T^3)$ with $ \int_{\T^3} v_{\initial}=0$,
\item[(c)] and a family $\{ v_\nu \}_{\nu >0}$ of (unique) smooth solutions of \eqref{e:NSE} and \eqref{e:Cauchy},
\end{itemize}
such that the following holds:
\begin{itemize}
\item[(i)] $\sup_{\nu \in [0,1]}   \| v_\nu \|_{L^{\infty } ((0,2) \times \T^3)} \leq 1$;
\item[(ii)] $\{v_\nu\}$ has at least two distinct limit points, as $\nu \to 0$, in the $L^\infty$ weak$^*$ topology, which are two distinct bounded admissible solutions $v_0^{\cc}$ and $v_0^{\ds}$ of \eqref{e:E} and \eqref{e:Cauchy};
\item[(iii)] $v_0^{\cc}$ satisfies \eqref{eq:conservative} while $v_0^{\ds}$ satisfies \eqref{eq:dissipative}.
\end{itemize}
\end{maintheorem}

Obviously the following are simple corollaries of the previous theorems.

\begin{corollary}[Non uniqueness for the forced Euler equations I]
Let $\alpha ' \in [0,1)$ be given. There exist $\sigma >0$, a body force $F_0 \in L^{1+ \sigma} ((0,2); C^\sigma (\T^3))$  and a divergence-free initial datum $v_{\initial } \in C^\infty (\T^3)$ such that  the $3d$ forced Euler equations \eqref{e:E}-\eqref{e:Cauchy} admit at least two distinct admissible bounded solutions. Furthermore, one of which belongs $L^{1}((0,2); C^{\alpha '} (\T^3))  $.
\end{corollary}

\begin{corollary}[Non uniqueness for the forced Euler equations II]
Let $\alpha ' \in [0,1)$ be given. There exist a body force $F_0 \in C^{\alpha ' } ((0,2) \times \T^3))$  and a divergence-free initial datum $v_{\initial } \in C^\infty (\T^3)$ such that  the $3d$ forced Euler equations \eqref{e:E}-\eqref{e:Cauchy} admit at least two distinct admissible bounded solutions.
\end{corollary}
 
We remark that, with a totally different method, Vishik in \cites{VishikI, VishikII} has produced nonuniqueness examples for the incompressible Euler equations in $\mathbb R^2$ in vorticity formulation when the solutions have vorticity in $C ([0,T], L^\infty \cap L^p)$ for any fixed $p<\infty$, while the curl of the body force belongs to $L^{1+\sigma} ([0,T], L^p)$ (cf. the lecture notes \cite{ABCDGJK}). In particular, using classical Calderon-Zygmund estimates, one can easily see that the velocities of these solutions belong to $C ([0,T], W^{1,p}_{\rm loc})$, while the body forces belong to~$L^{1+\sigma} ([0,T], W^{1,p}_{\rm loc})$. In fact Vishik's techniques have been successfully transposed to even show nonuniqueness of Leray solutions of the forced Navier-Stokes equations {\em at a fixed positive viscosity~$\nu>0$}, see \cite{ABC}. 

While the nature of the nonuniqueness results in \cites{VishikI, VishikII,ABCDGJK,ABC} is quite different from the constructions of this paper, they also strongly suggest that all the results of this section are likely to hold for body forces $\{F_\nu\}$ enjoying uniform bounds in $L^1 ([0,T], W^{1,p})$ and solutions of \eqref{e:NSE} enjoying uniform bounds in $L^\infty ([0,T], W^{1,p})$. They also suggest that the following question has likely a positive answer.

\begin{OQ}
Can the lack of selection of Theorems \ref{t_selection2} and \ref{t_selection} be shown with a $\nu$-independent force $F \in L^1( (0,2) ; L^\infty (\T^3))$ replacing the family 
$\{F_\nu\}_{\nu >0}$ (and $\{v_\nu\}_{\nu >0}$ a family of Leray solutions of~\eqref{e:NSE}-\eqref{e:Cauchy})?
\end{OQ} 

\subsection*{Acknowledgments}
EB is supported by the Giorgio and Elena Petronio Fellowship at the Institute for Advanced Study. MC and MS were supported by the SNSF Grant 182565 and by the Swiss State Secretariat for Education, Research and lnnovation (SERI) under contract number M822.00034. GC has been partially supported by the ERC Starting Grant 676675 FLIRT.

\section{Strategy of the proof} \label{section:heuristics}

We use the same strategy as in~\cites{BDL22,JY20,JYo20} and consider a $2 + \sfrac{1}{2}$-dimensional Navier-Stokes solution, for which the evolution decouples into a forced $2d$-Navier-Stokes system and a scalar advection-diffusion equation. The solution $v_\nu$ of the forced $2d$-Navier-Stokes system is a suitable regularization of the two-dimensional velocity field $u : [0,2] \times \T^2 \to \R^2$ constructed  in~\cite{CCS22}*{Section~4}, which is an alternating shear flow, that is, for every $t \in (0,1)$ we have either $u(t, x_1, x_2) = (W(t, x_2), 0)$ or $u(t, x_1, x_2) = (0, W(t, x_1))$. The third component of the $3d$-Navier-Stokes solution solves an advection-diffusion equation and will exhibit anomalous dissipation. 

More precisely, we define the solution of the forced $3d$-Navier-Stokes system and the initial condition as
$$
v_\nu = \begin{pmatrix}
 u_\nu
\\
\tilde \vartheta_\nu
\end{pmatrix}\, , 
\qquad
v_{\initial} = \begin{pmatrix}
0 
\\ 
\vartheta_{\initial} 
\end{pmatrix}
\, ,
$$
where $ u_\nu$ is a suitable regularization of $u$ (to be defined in Section~\ref{sec:solution 3d NS}) and $\tilde \vartheta_{\nu}$ solves the advection-diffusion equation with velocity field $ u_\nu$ and initial datum $\vartheta_{\rm in}$, i.e.
\begin{align*}
\begin{cases}
\partial_t \tilde \vartheta_\nu +  u_\nu  \cdot \nabla \tilde \vartheta_\nu = \nu \Delta \tilde \vartheta_\nu,
\\
\tilde \vartheta_{\nu} (0, \cdot ) = \vartheta_{\initial} (\cdot )\, .
\end{cases}
\end{align*}
Since $u_\nu$ is also an alternating shear flow (see Section~\ref{sec:solution 3d NS}, the nonlinear term $ u_\nu \cdot \nabla  u_\nu$ vanishes identically and therefore the velocity field $v_\nu$ solves the forced $3d$-Navier-Stokes system with force
$$ 
F_\nu  = \begin{pmatrix}
\partial_t  u_\nu - \nu \Delta  u_\nu 
\\
0
\end{pmatrix}\, .
$$

By suitably setting the parameters in the construction of $u$, we will verify that 
\begin{equation}
    v_\nu \in L^3((0,1); C^\alpha(\T^3))\, , 
    \quad
    F_\nu \in L^{1+\sigma}((0,1); C^\sigma(\T^3))
    \qquad \text{uniformly in $\nu$,}
\end{equation}
for some $\sigma>0$, where $\alpha < 1/3$ is arbitrary. In order to show that $v_\nu$ exhibits anomalous dissipation, hence concluding the proof of Theorem~\ref{t_Onsager}, we employ \cite{CCS22}*{Theorem~A} to get
\begin{equation} \label{anomalous}
   \limsup_{\nu \downarrow 0} 2 \, \nu \int_0^1 \int_{\T^3} |\nabla v_\nu(s,x)|^2\, dx\, ds 
    \ge 
    \limsup_{\nu \downarrow 0}
    2 \, \nu \int_0^1 \int_{\T^3} |\nabla \tilde \vartheta_\nu(s,x)|^2\, dx\, ds > 1/2\, .
\end{equation}

To prove that the vanishing viscosity limit does not select a unique solution in the setting of
Theorem~\ref{t_selection2} and Theorem~\ref{t_selection} we use the corresponding statement in~\cite{CCS22}*{Theorem~B} which proves lack of selection for solutions of the advection-diffusion equations with velocity field $u$. More precisely, we prove that the first two components of $v_\nu$ (namely $u_\nu$) strongly converge in $L^2((0,2) \times \T^3)$ to a unique limit whereas the last component of $v_\nu$ (namely $\tilde \vartheta_\nu$) for a suitable choice of a sequence of viscosity parameters $\{ \tilde \nu_q \}_{q \in \N}$ exhibits anomalous dissipation~\eqref{anomalous} and for another suitable choice of a sequence of viscosity parameters~$\{ \nu_q \}_{q \in \N}$ converges strongly in $L^2((0,2) \times \T^3)$ to a conservative solution (i.e.~the limit satisfies the energy balance~\eqref{eq:conservative} with the first two components of the velocity field). 



\section{Construction and main properties of the \texorpdfstring{$2d$}{2d} velocity field}\label{section:parameters}

In this section we recall the main properties of the velocity field $u : [0,1] \times \T^2 \to \R^2$ constructed in~\cite{CCS22} and of the corresponding solution $\vartheta_\nu : [0,1] \times \T^2 \to \R$ of the advection-diffusion equation with velocity field $u$. This velocity field will be used as a building block for the construction of solutions to the forced $3d$-Navier-Stokes equations in Theorems~\ref{t_Onsager},~\ref{t_selection2} and~\ref{t_selection}.



\subsection{Choice of the parameters}\label{ss:para}

Let $\alpha \in (0,1)$ and $\beta  \in [0,\sfrac{1}{3})$ such that $\alpha + 2 \beta <1$.
We consider parameters $\epsilon , \delta \in (0,\sfrac14)$ sufficiently small such that 
\begin{subequations} \label{c:beta_eps_condition_all}
\begin{align} 
1 - \frac{2 \beta (1 + 3 \epsilon(1+ \delta) ) (1+ \delta)}{1 -  \delta} - \alpha (1 + \epsilon \delta)(1 + \delta)  - \frac{\delta}{8 }>0\,, \label{c:alpha_beta_eps_kappa}
\\
\frac{3 \beta (1 + 3\epsilon (1 + \delta)) (1 + \delta )}{1 -  \delta} + \frac{\delta}{8} < 1 \,, \label{c:gamma_eps}
\\
\epsilon \leq \frac{\delta^3}{ 50}\,.  \label{c:eps_delta}
\end{align} 
\end{subequations}
Furthermore we introduce the parameter $\gamma>0$ as
\begin{align} \label{d:gamma}
\gamma = \frac{3 \beta (1 + 3\epsilon (1 + \delta)) (1 + \delta ) }{1 - \delta} + \frac{ \delta}{8} <1 \,.
\end{align}

Given $a_0 \in (0,1)$ such that 
\begin{align} \label{c:d_0}
a_0^{\epsilon \delta^2} + a_0^{\sfrac{\epsilon \delta}{8} } \leq \frac{1}{20} \,,
\end{align}
we define
\begin{align}    \label{d:parameters}
  a_{q+1} = a_q^{1 + \delta}, \qquad
\lambda_q =  \frac{1}{2 a_q}.
\end{align}

\subsection{Construction of the velocity field}

Let us begin by introducing some notation. For any $f : [0,2] \times \T^2 \to \R^2$ we denote by $\supp_{T} (f)$ the temporal support of the function $f$, namely the projection on the time interval $[0,2]$ of the support of $f$. The precise definition is
$$
\supp_{T} (f) := \overline{ \{ t \in [0,2]: \text{ there exists } x \in \T^2 \text{ such that } f(t,x) \neq 0 \}} \, .
$$
Given $\{T_q \}_{q \in \N\cup\{-1\}}$, a decreasing sequence of non-negative numbers such that $T_{-1} = 1$ and $T_q \downarrow 0$ as~$q\to \infty$, we define the time intervals
$$ 
\mathcal{I}_q = [1- T_q, 1 - T_{q+1}] \, ,
\qquad \mathcal{J}_q = [1+T_{q+1}, 1 + T_{q}]  \, ,
\quad \text{for any $q\in \N \cup \{-1\}$}\, .
$$

The results below are taken from \cite{CCS22}.

\begin{prop} \label{proposition}
Let $ \alpha$, $\beta$, $\gamma$, $\epsilon$, $\delta$, and $\{ a_q \}_{q \in \N}$ as above. 
Then there exist a decreasing sequence of times
$\{T_q \}_{q \in \N\cup\{-1\}}$ satisfying $T_{-1} = 1$ and $T_q \downarrow 0$ as $q\to \infty$,
an initial datum $\vartheta_{\initial} \in C^\infty(\T^2)$ with $ \int_{\T^2} \vartheta_{\initial}=0$, and a divergence-free velocity field $u \in C^{\infty}_{\rm loc}(  ((0,2) \setminus \{ 1\}) \times \T^2 ; \R^2 )$, such that the following hold:

\begin{enumerate}

\item{(Reflection and shear flow)}\label{vectorfield_VV}
For any $t\in (0,2)$, $u(t,\cdot)$ coincides either with an horizontal shear flow, or with a vertical one. Moreover
$u(t, x) = - u(2-t,x)$ for any $t \in (1,2)$ and $x \in \T^2$.

\item{(Time intervals)}\label{prop:propertysupport} For any $q \in \N$ we have $| T_q - T_{q+1}| \leq 4 a_q^{\gamma - \gamma \delta}$, and
\begin{align} \label{p:support}
      \supp_T(u)\cap (\mathcal{I}_{-1}\cup \mathcal{J}_{-1}) &= \emptyset\, ,
      \\
     \left | \supp_T(u) \cap (\mathcal{I}_q \cup \mathcal{J}_q) \right | &\leq 6 a_q^\gamma .
\end{align}
Moreover, $u(t,\cdot) \equiv 0$ for any $t$ in a neighborhood of $1-T_q$ and $1+T_q$.

\item{(Regularity of the velocity field)} \label{prop:property1} For any $k \in \N$ and $\ell \in \N$ there exists a constant $C>0$ such that 
\begin{align} \label{s:vectorfield_I_q5_new}
  \| \partial_t^\ell \nabla^k u \|_{L^\infty (( {\mathcal{I}}_{q} \cup \mathcal{J}_{q}) \times \T^2)} \leq C a_q^{1 - \gamma} a_{q+1}^{-k(1 +  \epsilon \delta )} a_q^{- \ell \gamma}  \,,
\end{align}
for any $q\in \N$.

\item{(Regularity of the solution)}\label{prop:regularity-solution} 
For any $\nu>0$ there exists a unique bounded solution $\vartheta_\nu : [0,2] \times \T^2 \to \R$ of the advection-diffusion equation  
\begin{align} \label{e:advection-diffusion}
\partial_t \vartheta_\nu + u \cdot \nabla \vartheta_\nu = \nu \Delta \vartheta_\nu 
\end{align}
with initial datum $\vartheta_{\initial}$. For $\nu=0$, the advection equation (i.e.,~\eqref{e:advection-diffusion} with $\nu =0$) with velocity field~$u$ and initial datum $\vartheta_{\initial}$ has a unique bounded solution with the symmetry $\vartheta_0 (t,x) = \vartheta_0 (2-t, x)$ for any $t \in (1,2)$ and $x \in \T^2$. The family of solutions $\{\vartheta_\nu\}_{\nu \in [0,1]}$ satisfies
$$
\sup_{\nu \in [0,1]} \| \nabla \vartheta_\nu \|_{L^{\infty} (\mathcal{I}_q \times \T^2)} \leq \| \nabla \vartheta_{\initial}\|_{L^\infty} a_{q+1}^{-1 -3 \epsilon (1+ \delta)}\, ,
\quad \text{for any $q\in \N$.}
$$

\item{(Anomalous dissipation)}\label{prop:anomalous} 
For any $q\in \N$ we set
\begin{equation} \label{eq: tilde nu}
    {\tilde \nu}_q = a_q^{2 - \frac{\gamma}{1 + \delta} + 4 \epsilon} \, .
\end{equation}
There exists $m \in\N$ such that the sequence $\{\vartheta_{\tilde \nu_q} \}_{q \in \N}$  satisfies
\begin{equation}
    \, 2 \, \tilde{\nu}_q \int_0^{1- T_q + \overline{t}_q} \int_{\T^2} | \nabla \vartheta_{\tilde \nu_q}|^2 \, dx\,dt > \frac{1}{2}\, 
    \qquad \text{ for any $q \in m\N$,}
\end{equation}
where $\bar t_q \in (T_{q+1}, T_q)$ is a suitable intermediate time such that $\supp_{T}(u)\cap (1-T_q, 1-T_q + \bar t_q) = \emptyset$.

\end{enumerate}
\end{prop}

\begin{proof}
The velocity field with all the above properties
is obtained from the one constructed in~\cite{CCS22}*{Section~4} choosing $p=p^\circ=\sfrac{1}{3}$.
Properties~\eqref{vectorfield_VV} and~\eqref{prop:propertysupport} are a direct consequence of the construction in \cite{CCS22}*{Section~4}. Property~\eqref{prop:property1} is given in~\cite{CCS22}*{Remark~4.2}. 
 Property~\eqref{prop:regularity-solution} has been proved in~\cite{CCS22}*{Section~8}. Property~\eqref{prop:anomalous} has been proved in~\cite{CCS22}*{Section~7} and it is stated in~\cite{CCS22}*{Theorem~A}. 
\end{proof}

\section{Solution of the forced \texorpdfstring{$3d$}{3d}-Navier-Stokes and Euler equations}
\label{sec:solution 3d NS}

Let $ \alpha$, $\beta$, $\gamma$, $\epsilon$, $\delta$, and $\{ a_q \}_{q \in \N}$ be as in Section~\ref{ss:para}. We employ the velocity field $u$ and the initial condition $\vartheta_{\rm in}$ built in Proposition \ref{proposition} to produce $(v_\nu, p_\nu, F_\nu)$ a smooth solution to the forced $3d$-Navier-Stokes equations \eqref{e:NSE}-\eqref{e:Cauchy}.

For any $q\in \N$, we introduce the closed set $K_q = [0,1-T_q]\cup [1+T_q,2] $ and define
\begin{equation}\label{e:defuq}
     u_q (t,x ) = u(t,x) \mathbbm{1}_{K_q} (t)\, .
\end{equation}
We observe that $u_q$ is smooth for any $q \in \N$.

We consider the family of viscosity parameters ${\tilde \nu}_q$ defined in~\eqref{eq: tilde nu}.
For any $\nu \in (0,a_0^2)$ there exists $q\in \N$ such that $\nu \in ( \tilde \nu_{q+1}, \tilde \nu_q]$.
Let $\tilde \vartheta_{\nu} : [0,2] \times \T^3 \to \R$ be the unique smooth solution to the advection-diffusion equation \eqref{e:advection-diffusion} with diffusion parameter $\nu$, initial datum $\vartheta_{\initial}$, and velocity field $ u_q (t,x )$, 
i.e.
\begin{align*}
\partial_t \vt_{\nu} + {u}_{q} \cdot \nabla \vt_{\nu} = \nu \Delta \vt_{ \nu} \, .
\end{align*} 
We define smooth functions $F_{\nu}, v_\nu : [0,2] \times \T^3 \to \R^3$ and $p_\nu : [0,2] \times \T^3 \to \R$ as 
\begin{align*}
F_{\nu}(t,x) 
    & = \begin{pmatrix}
\partial_t  u_q (t,x) - \nu \Delta  u_q (t,x)
\\
0
\end{pmatrix}
\\
v_{\nu} (t,x ) & = \begin{pmatrix}
   u_q(t,x)
 \\
 \vt_{\nu} (t,x)
\end{pmatrix}
\\p_{\nu} &= 0\, .
\end{align*}
Finally, we set
\begin{align} \label{initialdatum_NS}
 v_{\initial } = \begin{pmatrix}
0 
\\
 \vartheta_{\initial}
\end{pmatrix}\, .
\end{align}
Given Proposition \ref{proposition}, the following lemma is immediately checked.

\begin{lemma}\label{lemma:3dNS}
   For any $\nu\in (0,a_0^2)$, given $F_{\nu}$ as above, $(v_\nu, p_\nu)$ is the unique smooth solution to \eqref{e:NSE} with initial datum $v_{\rm in}$. Moreover, at any time $t\in (0,2)$ the velocity field 
    $v_\nu$ is an alternating shear flow on the first two components, i.e. 
$$v_{\nu} (t,x)  = \begin{pmatrix}
w^1_{\nu} (t, x_2)
\\
0 
\\
* 
\end{pmatrix} \qquad \mbox{ or } \qquad  v_{\nu} (t,x)
= \begin{pmatrix}
0
\\
w^2_{\nu} (t, x_1)
\\
* 
\end{pmatrix}
$$
for suitable one-dimensional functions $w_{\nu}^i: [0,2] \times \T \to \R$, for $i=1,2$.   
\end{lemma}

At least formally, we expect $(v_\nu,p_\nu, F_\nu)$ to converge to a solution of the forced $3d$-Euler equations~\eqref{e:E}-\eqref{e:Cauchy}
when $\nu\downarrow 0$. We will prove in the next sections that this is the case under suitable assumptions and that 
\begin{align}\label{F0}
F_{0}(t,x) 
    & = 
    \begin{pmatrix}
\partial_t u (t,x) 
\\
0
\end{pmatrix}
\\
p_{0} &= 0\, .
\end{align}
The following lemma immediately follows from the regularity of $u$ in $(0,1)\times \T^2$ in Proposition \ref{proposition}. 

\begin{lemma}\label{lemma:Euler(0,1)}
    Let $u$, $\vartheta_0$ be as in Proposition \ref{proposition}, and let $F_0$ be as in \eqref{F0}.
    We have that
    \begin{equation}
        F_0\in C^\infty ((0,1)\times \T^3)\, .
    \end{equation}
    Moreover, 
    \begin{equation}
    v_{0} (t,x ) 
    := \begin{pmatrix}
    u (t,x)
   \\
    \vartheta_{0} (t,x)
    \end{pmatrix} \, ,
    \quad t\in (0,1), \, x\in \T^3\, ,
    \end{equation}
    is the unique smooth solution to \eqref{e:E}-\eqref{e:Cauchy} in $(0,1)\times \T^3$ with initial datum \eqref{initialdatum_NS}.
\end{lemma}

\begin{remark}
    We will see in the next sections that uniqueness for~\eqref{e:E}-\eqref{e:Cauchy} may fail past time $t=1$, where the singularity of $F_0$ appears.
\end{remark}

\section{Proof of Theorem~\ref{t_Onsager} and Theorem~\ref{t_selection2}} \label{section:proofOnsager}

Let $\alpha \in [0, 1/3)$ be fixed as in Theorem \ref{t_Onsager} and $\alpha ' \in [0,1)$ be fixed as in Theorem \ref{t_selection2}. Without loss of generality and up to increasing $\alpha$ or $\alpha'$, we can assume $\alpha ' = 3 \alpha$. We fix $\beta = \alpha$ and choose the parameters~$\epsilon$,~$\delta$, and $\{ a_q \}_{q \in \N}$ as in Section \ref{ss:para}. The parameter $\gamma$ is then determined by~\eqref{d:gamma}. The viscosity parameter ${\tilde \nu}_q$ has been chosen in~\eqref{eq: tilde nu}. Let $(v_\nu, p_\nu, F_\nu)$  be the solution to \eqref{e:NSE}, with initial datum as in \eqref{initialdatum_NS}, built in Section~\ref{sec:solution 3d NS}. 

In order to prove Theorem~\ref{t_Onsager} and Theorem~\ref{t_selection2} we need to show the following facts:

\begin{itemize}

\item[(i)] There exists $\sigma>0$ such that
\begin{align} 
\sup_{\nu \in (0,a_0^2)}
\| v_\nu \|_{L^{3} ([0,1]; C^\alpha(\T^3))}  
+  \| v_\nu \|_{L^{\infty } ([0,2] \times \T^3)}
+ \| F_\nu \|_{L^{1+ \sigma} ([0,2]; C^{\sigma} (\T^3))}
< \infty \,.
 \end{align}
Moreover, $F_\nu \to F_0$ in $L^{1+\sigma} ([0,2]; C^{\sigma} (\T^3))$.

\item[(ii)] Let $v_0$ be as in Lemma \ref{lemma:Euler(0,1)}. We have that $v_\nu \to v_0$ in $L^2((0,1)\times \T^3)$ as $\nu \to 0$.

\item[(iii)] There exist $v_0^{\rm ds}\in L^\infty([0,2]\times \T^3)$ solution to \eqref{e:E} with initial datum \eqref{initialdatum_NS} and a sequence $q_k\to \infty$, such that
$v_{\tilde \nu_{q_k}} \to v_0^{\rm ds}$ weakly in $L^2([0,2]\times \T^3)$. Moreover,
\begin{equation} 
2 \, \tilde \nu_{q_k} \int_0^1 \int_{\T^3} | \nabla v_{\tilde \nu_{q_k}}|^2 \, dx\,dt > 1/2 
\qquad \text{for any $k\in\N$.}
\end{equation}
In particular $v_0^{\rm ds}$ is an admissible dissipative solution of \eqref{e:E}.

\item[(iv)] Set $\nu_q =a_q^{2 - \gamma + \delta + 8 \epsilon}$.
There exists $v_0^{\rm cs}\in L^1((0,2); C^{\alpha'}(\T^3))$, a conservative (admissible) solution to~\eqref{e:E} with initial datum \eqref{initialdatum_NS}, such that $v_{\nu_q} \to v_0^{\rm cs}$ strongly in $L^2((0,2)\times \T^3)$ as $q \to \infty$.
\end{itemize}

\medskip

\noindent{\bf Proof of (i).} From~\eqref{s:vectorfield_I_q5_new} and the maximum principle for the advection-diffusion equation (using that the initial datum is bounded), we deduce that 
$$ 
\sup_{\nu \in (0,a_0^2)}   \| v_\nu \|_{L^{\infty } ([0,2] \times \T^3)} < \infty \, .
$$
Let us now check that
\begin{align} 
\sup_{\nu \in (0,a_0^2)}
\| v_\nu \|_{L^{3} ([0,1]; C^\alpha(\T^3))}  
< \infty \, .
 \end{align}
This is a consequence of 
\begin{equation}\label{z5}
u \in L^3 ((0,1); C^{\alpha}(\T^3))
\qquad\text{ and }\qquad
\sup_{\nu \in (0,a_0^2)} \| \vartheta_\nu \|_{L^{3} ((0,1); C^\alpha(\T^2))}  < \infty 
\end{equation}
that we now prove. 
Indeed, $\tilde \vartheta_\nu \equiv \vartheta_\nu$ in $[0, 1 -T_q] \times \T^3$ since $ u_q \equiv u$ in $[0, 1 -T_q] \times \T^3$, while~$\tilde \vartheta_\nu(t,\cdot)$ solves the heat equation for $t \in [1-T_q, 1]$, and the H\"older norm is nonincreasing for solutions of the heat equation. 

Let us begin by proving the first property in~\eqref{z5}. By~\eqref{p:support} and~\eqref{s:vectorfield_I_q5_new} and by interpolation we have
\begin{align*}
\| u \|_{L^3((0,1);C^\alpha)}^3 & = \sum_{q=0}^\infty \int_{{\mathcal{I}}_q} \|u (s,\cdot)\|_{C^\alpha(\T^2)}^{3}  ds 
\leq  \sum_{q=0}^\infty \int_{{\mathcal{I}}_q} \|u (s,\cdot)\|_{L^\infty(\T^2)}^{3 (1- \alpha)}   \|u (s,\cdot)\|_{W^{1,\infty}(\T^2)}^{3 \alpha}   ds 
\\
& \le C\sum_{q=0}^\infty a_q^{ \gamma}  a_q^{-3(1- \alpha)(\gamma -1)  }  a_q^{- 3 \alpha (\gamma -1)  } a_{q+1}^{-3 \alpha (1 + \epsilon \delta)} 
\end{align*}
and the sum is finite if and only if 
$$ \frac{\gamma}{3} + 1 - \gamma - \alpha (1 + \epsilon \delta) (1 + \delta) >0\, ,
$$
which holds thanks to the choice \eqref{d:gamma} and the condition \eqref{c:alpha_beta_eps_kappa}.

Let us show the second property in~\eqref{z5}. 
Fix $\nu \in (0,a_0^2)$ and correspondingly let $q\in \N$ such that $\nu \in ( \tilde \nu_{q+1}, \tilde \nu_q]$.
Thanks to property \eqref{prop:regularity-solution} of Proposition \ref{proposition} and using $\alpha = \beta$, we get
 \begin{align*}
\| \vartheta_\nu \|_{L^{3 }((0,1); C^\alpha)}^{3} 
&= \sum_{q=0}^\infty \int_{{\mathcal{I}}_q} \| \vartheta_\nu (s, \cdot ) \|_{ C^\beta}^{3}  ds 
\le  C\sum_{q=0}^{\infty} a_q^{ \gamma -  \gamma \delta} a_{q+1}^{-3( \beta + 3\beta  \epsilon (1+ \delta))} 
\\
& =  C\sum_{q=0}^{\infty} a_q^{ \gamma -  \gamma \delta - 3( \beta + 3\beta  \epsilon (1+ \delta))(1 + \delta)} 
\end{align*}
and the sum is finite and independent of $\nu$ since as a consequence of \eqref{d:gamma} we have
$$
- \gamma (1 -  \delta) + 3( \beta + 3\beta  \epsilon (1+ \delta))(1 + \delta) <0 \,.
$$

We finally prove that
\begin{equation}
    F_\nu \in L^{1+\sigma}((0,2); C^\sigma(\T^3))\, , \quad \text{uniformly in $\nu\in (0,a_0^2)$}\, ,
\end{equation}
for some $\sigma>0$, and $F_\nu \to F_0$ in $L^{1+\sigma}((0,2); C^\sigma(\T^3))$ as $\nu \to 0$.
To this aim, it is enough to show that there exists $C>0$ such that
for any $ \nu \in (\tilde \nu_{q+1}, \tilde  \nu_q]$ we have
\begin{equation}\label{e:regforce1}
\| \partial_t u \|_{L^{1+ \sigma}((0,2); C^\sigma ( \T^3))} \leq C 
\qquad  \text{ and } \qquad
\| \nu \Delta u \|_{L^{1+ \sigma}((K_q; C^\sigma( \T^3))} \leq C a_q^\epsilon \, \, .
\end{equation}
and that
\begin{equation}\label{e:zeroforce1}
\| \partial_t u \|_{L^{1+ \sigma}(K_q^c; C^\sigma ( \T^3))} \to 0
\qquad \text{ as $q \to \infty$.}
\end{equation}
For the first property in~\eqref{e:regforce1}, thanks to \eqref{s:vectorfield_I_q5_new} we have
\begin{align*}
\| \partial_t u \|_{L^{1+ \sigma}((0,2); C^\sigma( \T^3))}^{ 1+ \sigma} 
& \leq \sum_{j=0}^{\infty} \int_{\mathcal{I}_j\cup \mathcal{J}_j}  \left ( \| \partial_t u (s, \cdot) \|_{L^\infty (\T^3)}^{(1- \sigma)} \| \partial_t u (s, \cdot) \|_{W^{1,\infty} (\T^3)}^{\sigma} \right )^{1+ \sigma} ds
\\
& \leq C  \sum_{j=0}^{\infty} \int_{\mathcal{I}_j\cup \mathcal{J}_j}  \left ( a_j^{1-  2 \gamma } a_{j+1}^{- \sigma (1+ \epsilon \delta)} \right )^{1+ \sigma} ds 
\\&
\leq 4 C \sum_{j=0}^{\infty}  a_j^{\gamma} a_j^{(1+ \sigma)(1- 2  \gamma - \sigma (1+ \delta) (1+ \epsilon \delta ))  } < \infty
\end{align*} 
where we used that $1- \gamma >0 $, and we choose $\sigma>0$ sufficiently small such that $\gamma + (1+ \sigma)(1-  2 \gamma - \sigma (1+ \delta) (1 + \epsilon \delta ) )>0$. Property~\eqref{e:zeroforce1} follows by noticing that
$$
\| \partial_t u \|_{L^{1+ \sigma}(K_q^c; C^\sigma( \T^3))}^{ 1+ \sigma} 
\leq \sum_{j=q}^{\infty} \int_{\mathcal{I}_j\cup \mathcal{J}_j}  \left ( \| \partial_t u (s, \cdot) \|_{L^\infty (\T^3)}^{(1- \sigma)} \| \partial_t u (s, \cdot) \|_{C^1 (\T^3)}^{\sigma} \right )^{1+ \sigma} ds
\quad \to \quad  0 
\qquad \text{ as $q \to \infty$.}
$$
For the second property in~\eqref{e:regforce1}, thanks to \eqref{s:vectorfield_I_q5_new}, we have
\begin{align*}
\| \nu \Delta u \|_{L^{1+ \sigma}(K_q; C^\sigma( \T^3))}^{1+ \sigma} & \leq \tilde \nu_q \sum_{j=0}^{q-1} \int_{\mathcal{I}_j\cup \mathcal{J}_j} \left (  \| \Delta u (s, \cdot) \|_{L^\infty (\T^3)}^{1- \sigma } \| \Delta u (s, \cdot) \|_{W^{1,\infty} (\T^3)}^{ \sigma } \right )^{1+ \sigma} ds
\\
& \leq C a_q^{2 - \frac{\gamma}{1 + \delta} + 4 \epsilon}  \sum_{j=0}^{q-1} \int_{\mathcal{I}_j\cup \mathcal{J}_j}  \left ( a_j^{1- \gamma} a_{j+1}^{-2 (1 +  \epsilon \delta) (1- \sigma) }  a_{j+1}^{-3 (1 +  \epsilon \delta) \sigma  } \right)^{1+ \sigma}  ds 
\\
& \leq C  a_q^{2 - \frac{\gamma}{1 + \delta} + 2 \epsilon}  \sum_{j=0}^{q-1}   a_j^\gamma \left ( a_j^{ - 1  - 2 \delta - \gamma}   a_j^{- \sigma (1+ \epsilon \delta )(1+ \delta) } \right )^{1 + \sigma}
\\
& \leq C q a_q^{2 \epsilon} a_{q-1}^{2 + 2 \delta - \gamma} a_{q-1}^{\gamma}  a_{q-1}^{- (1+\sigma)(1 + 2 \delta + \gamma + \sigma (1+ \epsilon \delta )(1+ \delta) )} 
\leq C a_q^{\epsilon} ,
\end{align*}
where we used that $q a_q^{\epsilon} \leq 1$, $a_{j+1} = a_j^{1+ \delta}$, $1 - \gamma >0$ and we choose $\sigma >0$ sufficiently small to guarantee that~$2 + 2 \delta - (1+ \sigma)( 1 + 2 \delta + \gamma + \sigma (1+ \epsilon \delta )(1+ \delta)  ) >0$.

\medskip

\noindent{\bf Proof of (ii).} Recalling~\eqref{e:defuq}, it suffices to prove that $\vt_{\nu} \to \vartheta_0$ in $L^2((0,1) \times \T^3)$, as $\nu \to 0$.
Fix~$\nu \in (0,a_0^2)$, and let $q\in \N$ such that $\nu \in (\tilde \nu_{q+1}, \tilde \nu_q]$.
We employ a standard vanishing viscosity estimate. For any~$0 \le t\le 1-T_q$, we have that $ u_q = u$, hence
$$
\partial_t (\vt_{\nu} -  \vartheta_0) + u \cdot  \nabla (\vt_{\nu} - \vartheta_0) = \nu \Delta \vt_\nu \, \qquad \text{ for any } 0 \leq t \leq 1-T_q.
$$
We multiply the above equation by $\vt_{\nu} - \vartheta_0$ and integrate in space-time to get
\begin{align*}
\| \vt_{\nu} (t, \cdot) - \vartheta_0 (t, \cdot) \|_{L^2 (\T^3)}^2 
& \leq \nu \left | \int_0^t \int_{\T^3} \nabla \vt_\nu (s, x) \cdot \nabla \vartheta_0 (s,x) dx ds \right |
\\
& \leq   
\left( \nu \int_0^t \int_{\T^3} \left | \nabla \vt_\nu (s, x) \right |^2 dx ds \right )^{1/2}  \left( \nu \int_0^t \int_{\T^3} \left | \nabla \vartheta_0 (s, x) \right |^2 dx ds \right )^{1/2}\, .
\end{align*}
We observe that by the energy equality
\begin{equation}\label{z2}
\nu \int_0^t \int_{\T^3} \left | \nabla \vt_\nu (s, x) \right |^2 dx ds   \leq 1 \qquad \text{for any $t\in [0,1]$}\, .
\end{equation}
Let us define $t(\nu):= 1 - T_{k(q)} \le 1 - T_q$, where $k(q)$ is the largest natural number satisfying 
$$ 
 a_q^{2- \frac{\gamma}{1+ \delta}}  \exp \left( a_{k(q)}^{2-2\gamma} a_{k(q) + 1}^{-2 - 2\epsilon\delta } \right)  \leq 1\, .
$$
We claim that 
\begin{equation}\label{z3}
\nu \int_0^{t(\nu)} \int_{\T^3} \left | \nabla \vartheta_0 (s, x) \right |^2 dx ds \to 0 
\qquad \text{ and } \qquad
t(\nu) \to 1
\qquad \mbox{ as } \nu \to 0\, .
\end{equation}
This follows by Gr\"onwall inequality and~\eqref{s:vectorfield_I_q5_new}, since
\begin{align*}
\nu \int_0^{t(\nu)} \| \nabla \vartheta_0 (s, \cdot) \|_{L^\infty (\T^3)}^2 ds 
& \leq \nu  \| \nabla \vartheta_{\rm in} (\cdot)\|_{L^\infty(\T^3)}^2 \exp \left  (\int_0^{t(\nu)} \| \nabla u (s, \cdot ) \|_{L^\infty}^2 \right ) ds 
\\
& \leq 
C  a_q^{2- \frac{\gamma}{1+ \delta} + 4 \epsilon}   a_0^{-2 - 2\epsilon \delta} \exp \left( a_{k(q)}^{2-2\gamma} a_{k(q) + 1}^{-2 -2 \epsilon\delta }\right) 
\\
& \leq C a_q^{4 \epsilon} a_0^{-2 - 2\epsilon \delta} \to 0 
\end{align*} 
as $q \to \infty $. 
Finally, $t(\nu)\to 1$ as $\nu \to 0$ follows by the fact that $k(q)\to \infty$ as $q\to \infty$.

Building upon \eqref{z2}, \eqref{z3}, and the fact that $\vt_\nu$ and $\vartheta_0$ are uniformly bounded by $1$, we deduce
\begin{equation}
     \|\vt_{\nu} - \vartheta_0\|_{L^2 ((0,1) \times \T^3)}^2
     \le 
     \nu \int_0^{t(\nu)} \| \nabla \vartheta_0 (s, \cdot) \|_{L^\infty (\T^3)}^2 ds
     +
     C(1-t(\nu)) \to 0
\end{equation}
as $q\to \infty$.

\medskip 

\noindent{\bf Proof of (iii).}\label{subsec:ThAC(iii)}
  We observe that the sequence of solutions $\vartheta_{\tilde \nu_q}$ of the advection-diffusion equation with diffusion parameter $\tilde \nu_q$, velocity field $u$, and initial datum $\vartheta_{\initial}$ satisfies
  \begin{equation}\label{z6}
  2 \,    \tilde \nu_q \int_0^{1-T_q + \bar t_q} \| \nabla  \vartheta_{\tilde \nu_q}(s,\cdot)\|_{L^2}^2 \, ds > \frac{1}{2}
  \qquad \text{for any $q \in m\N$,}
  \end{equation}
  as a direct consequence of 
 \eqref{prop:anomalous} in Proposition \ref{proposition}.
Therefore $\vt_{\tilde \nu_q}$, the third component of $v_{\tilde \nu_q}$, satisfies~\eqref{z6} as well since $\vartheta_{\tilde \nu_q} = \tilde \vartheta_{\tilde \nu_q}$ in $[0,1-T_q + \bar t_q]$.

The first two components of $v_\nu$  strongly converge to $ u$ in $ L^{\infty} ((0,2 ) \times \T^3)$ since $ \| {u} \|_{L^\infty (K_q^c \times \T^3)} \leq 2 a_q^{1- \gamma} \to 0$ as $q \to \infty$.
It is simple to see that $\{\tilde \vartheta_{\tilde \nu_q}\}_{q\in \N}$ admits limit points in the weak topology of $L^2((0,2)\times \T^2)$ and that any such limit point solves the transport equation with velocity field $u$ and initial datum $\vartheta_{\rm in}$. Let us fix a limit point and denote it by $\vartheta^{\rm ds}$. It follows by~\eqref{z6} that $\vartheta^{\rm ds}$ is a dissipative solution of the transport equation.
We define 
\begin{equation}
    v_{0}^{\rm ds} (t,x ) 
    := \begin{pmatrix}
    u (t,x)
   \\
    \vartheta^{\rm ds} (t,x)
    \end{pmatrix} \, ,
    \quad t\in (0,2), \, x\in \T^3\, ,
\end{equation}
and check that $(v_0^{\rm ds}, p_0, F_0)$ with $p_0=0$ solves \eqref{e:E}. Indeed, since the first two components of $v_{\tilde \nu_q}$ strongly converge to $ u$ in $ L^{\infty} ((0,2 ) \times \T^3)$ and the last component converges weakly to $\vartheta^{\rm ds}$, the quadratic term~$v_{\tilde \nu_q} \cdot \nabla v_{\tilde \nu_q}$ converges in the sense of distributions to $v_0^{\rm ds} \cdot \nabla v_0^{\rm ds}$. It is straightforward to check that all the other terms in the distributional formulation of \eqref{e:E} pass to the limit as $\tilde \nu_q \to 0$. Finally the admissibility condition~\eqref{admissible} of $v_0^{\rm ds}$ follows from the fact that it is a weak* limit in $L^\infty$ of admissible solutions $v_\nu$ with force $F_\nu$ and the forces $F_\nu $ are  strongly converging to $F_0$ in $L^1$.

\medskip

\noindent{\bf Proof of (iv).} 
Let $\nu_q =a_q^{2 - \gamma + \delta + 8 \epsilon} \in (\tilde \nu_{q+1} , \tilde \nu_{q}]$. As before we have 
$$ v_{\nu_q} (t,x)  = \begin{pmatrix}
 u (t,x) \mathbbm{1}_{K_q} (t)
 \\
 \vt_{\nu_q} (t,x)
\end{pmatrix}  \,.
$$
Recalling the proof of (iii), we only need to prove that the last component of $v_{\nu_q}$ strongly converges in $L^2((0,2) \times \T^3)$ to a velocity field $v_0^{\rm cs}\in L^1((0,2); C^{\alpha'}(\T^3))$ that conserves in time the spatial $L^2$ norm and the admissibility condition \eqref{admissible} will directly follow from the conservative property.

We show that $\tilde \vartheta_{\nu_q} \to \vartheta_{0}$ in $L^2((0,2) \times \T^3)$, 
where $\vartheta_0$ is the symmetric solution to the transport equation in Proposition \ref{proposition}(4). To this aim, it is enough to show that 
$\| \tilde \vartheta_{\nu_q} (t, \cdot) - \vartheta_{q} (t, \cdot )\|_{L^2(\T^2)} \to 0$ as~$q \to \infty $ for any $t \in (0,2)$, where $\vartheta_q$ is the unique solution of the transport equation with velocity field $ u_q$ and initial datum $\vartheta_{\rm in}$. Indeed, this will entail
$$
\| \tilde \vartheta_{\nu_q} - \vartheta_0 \|_{L^2 ([0,2] \times \T^2)}  \leq  \| \tilde \vartheta_{\nu_q} - \vartheta_q \|_{L^2([0,2] \times \T^2)} + \| \vartheta_{q} - \vartheta_0 \|_{L^2([0,2] \times \T^2)} \to 0 \qquad \mbox{ as } q \to \infty \, ,
$$
where the second term $\| \vartheta_{q} - \vartheta_0 \|_{L^2 ((0,2) \times \T^3)} \to 0$ as $q \to \infty$, thanks to $\vartheta_q (t,\cdot)= \vartheta_0 (t,\cdot)$ for any $t \in [1-T_q, 1+T_q ]^c$, and the $L^\infty$ bound $\| \vartheta_0 \|_{L^\infty ((0,2) \times \T^2)} + \| \vartheta_q \|_{L^\infty ((0,2) \times \T^2)} \leq 2  $.

For any $t \in (0,2)$, using a standard energy estimate with the regularity bound~\eqref{prop:regularity-solution} and the symmetry property~\eqref{vectorfield_VV} from Proposition~\ref{proposition}, we estimate
\begin{align*}
&\| \tilde \vartheta_{\nu_q} (t, \cdot) - \vartheta_q (t, \cdot ) \|_{L^2(\T^2)}^2 \\
\leq &
2 \nu_q \left | \int_0^t \int_{\T^2} \nabla \tilde \vartheta_{\nu_q} (s,x) \cdot \nabla \vartheta_{q}(s,x) dx ds \right | 
\\
 \leq & 2  \left ( \nu_q\int_0^t \int_{\T^2} | \nabla  \vartheta_{q} (s,x)|^2  dx ds \right )^{1/2} 
\\
 = & 2 \left (\nu_q \sum_{j = q}^\infty \int_{\mathcal{I}_j\cup \mathcal{J}_j} \int_{\T^2} | \nabla \vartheta_{q} (s,x)|^2  dx ds 
 +
 \nu_q \sum_{j = 0}^{q-1}\int_{\mathcal{I}_j\cup \mathcal{J}_j} \int_{\T^2} | \nabla \vartheta_{q} (s,x)|^2  dx ds 
 \right )^{1/2}
 \\
 \le &
   C \left(  
   a_q^{2-\gamma + \delta + 8\epsilon} a_q^\gamma a_q^{-2(1+3\epsilon(1+\delta))}
   +
   a_q^{2-\gamma + \delta + 8\epsilon}
   \sum_{j=0}^{q-1} a_j^{\gamma} a_{j+1}^{-2(1+3\epsilon(1+\delta))}  
   \right)^{1/2} 
\\
& \leq C  a_q^{1- \frac{\gamma}{2} + \frac{\delta}{2} + 4 \epsilon } a_{q-1}^{\gamma/2} a_q^{-1 -3 \epsilon (1+\delta)} 
\leq C a_q^{\frac{\delta}{2} - \frac{\delta \gamma}{2}} \to 0\, ,
\end{align*}
as $q \to \infty$,
where we used $\gamma < 1$, $a_{q} = a_{q-1}^{1+ \delta}$, and $\delta \in (0,\sfrac{1}{8})$.

We finally show that $v_0^{\rm cs}\in L^1((0,2); C^{\alpha'}(\T^3))$.
Using \eqref{p:support} and \eqref{s:vectorfield_I_q5_new} we deduce
\begin{align*}
\| u \|_{L^1((0,2); C^{\alpha ' })} & = 2 \sum_{q=1}^\infty \int_{{\mathcal{I}}_q} \|u (s,\cdot)\|_{C^{\alpha ' }(\T^2)}  ds \leq  2 \sum_{q=1}^\infty \int_{{\mathcal{I}}_q} \|u (s,\cdot)\|_{L^\infty(\T^2)}^{1- \alpha '  }   \|u (s,\cdot)\|_{W^{1,\infty}(\T^2)}^{ \alpha ' }   ds 
\\
& \le C\sum_{q=0}^\infty a_q^{ \gamma}  a_q^{ 1 - \gamma  } a_{q+1}^{- \alpha '  (1 + \epsilon \delta)} = C \sum_{q=0}^\infty    a_{q}^{1 - \alpha '  (1 + \epsilon \delta) (1+ \delta)}  < \infty
\end{align*}
recalling that $\alpha ' = 3 \alpha$ and $\alpha = \beta$, and the last inequality holds thanks to  the condition \eqref{c:alpha_beta_eps_kappa}, which implies $1 - 3 \alpha (1 + \epsilon \delta) (1+ \delta) >0$.
For the last component of $v_{0}^{\cc}$, namely $\vartheta_0$, we recall that
$$ 
\vartheta_0 (t,x) = \vartheta_0 (2-t, x) \qquad \mbox{for any } x \in \T^3 \mbox{ and } t \in (1,2]\, ,
$$
and that it solves the transport equation (namely~\eqref{e:advection-diffusion} with $\nu =0$) with velocity field $u$.
Therefore, it is sufficient to estimate $\vartheta_0$ in $[0,1] \times \T^2$. 
Using \eqref{prop:regularity-solution} in Proposition \ref{proposition} we have
\begin{align*}
\| \vartheta_0 \|_{L^1((0,1); C^{\alpha ' })}
&= \sum_{q=0}^\infty \int_{{\mathcal{I}}_q} \| \vartheta_0 (s, \cdot ) \|_{ C^{\alpha '}}  ds \leq 4 \| \nabla \vartheta_{\initial} \|_{L^\infty}  \sum_{q=0}^{\infty} a_q^{ \gamma -  \gamma \delta} a_{q+1}^{- \alpha ' (1+ 3  \epsilon (1+ \delta))} 
\\
& =4  \| \nabla \vartheta_{\initial} \|_{L^\infty}  \sum_{q=0}^{\infty} a_q^{ \gamma -  \gamma \delta -  \alpha ' (1+ 3\alpha '  \epsilon (1+ \delta))(1 + \delta)}  < \infty
\end{align*}
where the last estimate holds thanks to $\alpha= \alpha ' /3$, \eqref{c:alpha_beta_eps_kappa}, \eqref{d:gamma}, and $\epsilon < \frac{\delta}{16 (1+ \delta)^2 }$ (a consequence of~\eqref{c:eps_delta}).

\section{Proof of Theorem \ref{t_selection}} \label{section:proof_selection}

Let $\alpha ' \in [0,1)$ be as in Theorem \ref{t_selection}. We fix $\alpha = \alpha '$ and $\beta =0$ and we choose the parameters $\delta$, $\epsilon$, $\gamma$, and~$\{ a_q \}_{q \in \N}$ as in Section~\ref{ss:para}. The parameters satisfy \eqref{c:beta_eps_condition_all}, \eqref{d:gamma}, \eqref{c:d_0}, and the further condition
\begin{align} \label{c:add_b}
1  - \alpha ' (1 + \epsilon \delta)(1 + \delta)  - \frac{\delta}{4 } >0 \, 
\end{align}
which is compatible with all the other conditions.
Let $(v_\nu, p_\nu, F_\nu)$ be the solution to \eqref{e:NSE}, with initial datum as in \eqref{initialdatum_NS}, built in Section \ref{sec:solution 3d NS}.

In order to prove Theorem~\ref{t_selection} we need to show the following facts:
\begin{itemize}

\item[(i)] There holds
$$
\sup_{\nu \in (0,a_0^2)}
 \| v_\nu \|_{L^{\infty } ([0,2] \times \T^3)}
 + 
 \| F_\nu \|_{C^{\alpha'}((0,2)\times \T^3)}
< \infty \, .
$$
Moreover, $F_\nu \to F_0$ in $C^{\alpha'}((0,2)\times \T^3)$.


\item[(ii)] There exist $v_0^{\rm ds}\in L^\infty([0,2]\times \T^3)$ solution to \eqref{e:E} with initial datum \eqref{initialdatum_NS} and a sequence $q_k\to \infty$, such that
$v_{\tilde \nu_{q_k}} \rightharpoonup v_0^{\rm ds}$ weakly in $L^2([0,2]\times \T^3)$. Moreover,
\begin{equation} 
 2 \, \tilde \nu_{q_k} \int_0^1 \int_{\T^3} | \nabla v_{\tilde \nu_{q_k}}|^2 \, dx\,dt \geq 1/2 \, \qquad \text{for any } q_k.
\end{equation}
In particular $v_0^{\rm ds}$ is an admissible dissipative solution of \eqref{e:E}.

\item[(iii)] Set $\nu_q =a_q^{2 + 3 \epsilon}$.
There exists $v_0^{\rm cs}\in L^\infty((0,2) \times \T^3)$, an (admissible) conservative solution to \eqref{e:E} with initiald datum \eqref{initialdatum_NS}, such that $v_{\nu_q} \to v_0^{\rm cs}$ in $L^2((0,2)\times \T^3)$. 
\end{itemize}

\medskip

\noindent{\bf Proof of (i).}
 Since $u$ is bounded, more precisely $\| u (t, \cdot) \|_{L^\infty ((0,2) \times \T^3)} \leq 2 a_0^{1- \gamma} \leq 1$
 and $\| \tilde  \vartheta_\nu \|_{L^\infty ((0,2) \times \T^3)} \leq \| \vartheta_{\initial} \|_{L^\infty (\T^3)} = 1$ by the maximum principle, we have  
$$
 \| v_\nu \|_{L^\infty((0,2) \times \T^3) }  \leq 1\, .
$$
Let us now show the uniform-in-viscosity regularity of $F_\nu$. If suffices to prove that there exists $C>0$ such that for any $ \nu \in (\tilde \nu_{q+1}, \tilde  \nu_q]$ we have
\begin{equation}\label{e:regforce2}
\| \partial_t u \|_{C^{\alpha '}((0,2) \times \T^3 )} \leq C 
\qquad\text{ and } \qquad
\| \nu \Delta u \|_{C^\alpha(K_q \times \T^3 )} \leq C a_q^\epsilon  \,.
\end{equation} 
We estimate the first term. Thanks to \eqref{s:vectorfield_I_q5_new} and the interpolation inequality, we have 
 \begin{align*}
 \| \partial_t u \|_{C^{\alpha '} ((0,2) \times \T^3)} & \leq  \sup_{j \in \N} \| \partial_t u  \|_{L^\infty (\mathcal{I}_j; C^{\alpha'}( \T^3))}   +  \sup_{j \in \N} \| \partial_t u \|_{L^\infty (\T^3; W^{1,\infty}( \mathcal{I}_j ))}   
 \\
 & \leq C  \left ( \sup_{j \in \N} \| \partial_t u  \|^{1- \alpha' }_{L^\infty (\mathcal{I}_j; L^\infty( \T^3))} \| \partial_t u  \|^{ \alpha' }_{L^\infty (\mathcal{I}_j; W^{1,\infty}( \T^3))} + \sup_{j \in \N} a_j^{1- 2 \gamma}  \right) 
 \\
  & \leq C \sup_{j \in \N} a_j^{1- 2 \gamma} a_{j+1}^{- \alpha ' (1+ \epsilon \delta)} +1 < \infty
 \end{align*} 
where we used \eqref{c:add_b} and  $\gamma = \sfrac{\delta}{8} < \sfrac{1}{2}$. This proves the first property in~\eqref{e:regforce2}. In order to show the second property in~\eqref{e:regforce2}, we exploit~\eqref{s:vectorfield_I_q5_new} and tha fact that $\nu \in ( \tilde \nu_{q+1}, \tilde \nu_q]$ to obtain
 \begin{align*}
 \| \nu \Delta u \|_{C^{\alpha'}((0,1- T_q) \times  \T^3))} & \leq  \tilde \nu_q \sup_{j \leq q-1} \| \Delta u  \|_{C^{\alpha '} ( \mathcal{I}_j \times \T^3 )} 
 \\
 & \leq C a_q^{2 - \frac{\gamma}{1 + \delta} + 4 \epsilon} \sup_{j \leq q-1} ( a_{j}^{1- \gamma} a_{j+1}^{-2 -2 \epsilon \delta} a_{j+1}^{- \alpha ' (1+ \epsilon \delta)} +  
 a_{j}^{1- \gamma} a_{j+1}^{-2 -2 \epsilon \delta} a_j^{- \gamma})
 \\
 & \leq C a_q^{ \epsilon} a_{q-1 }^{ 1- 2 \gamma - \alpha ' (1 + \epsilon \delta)(1+ \delta)} \leq C a_q^\epsilon \,,
 \end{align*}
where we also used $a_q^{2 \epsilon} a_q^{- 2 \epsilon \delta } \leq 1$, \eqref{c:alpha_beta_eps_kappa}, and \eqref{c:add_b}.

The convergence $F_\nu \to F_0$ in $C^{\alpha'}((0,2)\times \T^3)$ can be shown along the same lines, by observing 
that~$\| \partial_t u \|_{C^{\alpha '} (K_q^c \times \T^2)} \to 0$ as $q \to \infty$.

\medskip 

\noindent{\bf Proof of (ii).}
We argue exactly as in the proof of~(iii) in Section~\ref{section:proofOnsager}.
We first notice that
$$
2 \, \tilde \nu_q \int_0^{1-T_q + \bar t_q} \| \nabla  \tilde \vartheta_{\tilde \nu_q}(s,\cdot)\|_{L^2}^2 \, ds
    =
      2 \, \tilde \nu_q \int_0^{1-T_q + \bar t_q} \| \nabla  \vartheta_{\tilde \nu_q}(s,\cdot)\|_{L^2}^2 \, ds > \frac{1}{2}
      \qquad \text{ for any $q \in m\N$,}
$$
as a direct consequence of~\eqref{prop:anomalous} in Proposition \ref{proposition}.
The first two components of $v_\nu$  strongly converge to $ u$ in $ L^{\infty} ((0,2 ) \times \T^3)$
while~$\{\tilde \vartheta_{\tilde \nu_q}\}_{q\in \N}$ admits a limit point $\vartheta^{\rm ds}$ in the weak topology of $L^2((0,2)\times \T^2)$ which solves the transport equation with velocity field $u$ and initial datum $\vartheta_{\rm in}$. Setting
\begin{equation}
    v_{0}^{\rm ds} (t,x ) 
    := \begin{pmatrix}
    u (t,x)
   \\
    \vartheta^{\rm ds} (t,x)
    \end{pmatrix} \, ,
    \quad t\in (0,2), \, x\in \T^3\, ,
\end{equation}
we can verify that $(v_0^{\rm ds}, p_0, F_0)$ with $p_0=0$ solves \eqref{e:E} and $v_0^{\rm ds}$ is an admissible solution by arguing exactly as in the proof of~(iii) in Section~\ref{section:proofOnsager}.


\medskip 

\noindent{\bf Proof of (iii).}
The first two components of $v_{\nu_q}$ strongly converge to $u$ in  $L^\infty ((0,2) \times \T^3)$.
We claim that~$\vt_{\nu_q}$, the last component of $v_{\nu_q}$, strongly converges in $L^2((0,2) \times \T^2)$ to $\vartheta_0$ (defined as in~\eqref{prop:regularity-solution} of Proposition~\ref{proposition}). Setting
\begin{equation}
    v_{0}^{\rm cs} (t,x ) 
    := \begin{pmatrix}
    u (t,x)
   \\
    \vartheta_0 (t,x)
    \end{pmatrix} \, ,
    \qquad t\in (0,2), \, x\in \T^3
\end{equation}
and observing that $\| \vartheta_0 (t,\cdot)\|_{L^2} = \| \vartheta_{\rm in}\|_{L^2}$ for any $t\in (0,2)\setminus\{1\}$, $F_{\nu_q}\to F_0$ in $C^{\alpha'}((0,2)\times \T^3)$, and $ u_q \to u$ in $L^2((0,2)\times \T^3)$, the claimed convergence suffices to conclude that $v_0^{\rm cs}$ is an (admissible) conservative solution to \eqref{e:E}.


We argue as in the proof of~(iv) in Section~\ref{section:proofOnsager}. Denoting by $\vartheta_q:(0,2)\times \T^2 \to \R$ the unique solution to the transport equation with velocity field $u_q$ and initial datum $\vartheta_{\rm in}$, we have
\begin{equation}\label{e:lastsplit}
 \|\tilde  \vartheta_{\nu_q} - \vartheta_0 \|_{L^2 ([0,2] \times \T^2)}  
 \leq  \| \tilde \vartheta_{\nu_q} - \vartheta_q \|_{L^2 ([0,2] \times \T^2)} + 
 \| \vartheta_{q} - \vartheta_0 \|_{L^2 ([0,2] \times \T^2)} \, .
\end{equation}
We notice that $\| \vartheta_{q} - \vartheta_0 \|_{L^2 ((0,2) \times \T^3)} \to 0$ as $q \to \infty$, thanks to $\vartheta_q (t,x)= \vartheta_0 (t,x)$ for any $t \in K_q$ and any $x \in \T^3$ (because of the symmetry of the velocity field $u$ as in property \eqref{vectorfield_VV} of Proposition \ref{proposition}) and to the bound $\| \vartheta_0 \|_{L^\infty ((0,2) \times \T^2)} + \| \vartheta_q \|_{L^\infty ((0,2) \times \T^2)} \leq 2  $. For any $t \in (0,2)$, we estimate the first term in~\eqref{e:lastsplit} relying on the regularity bound~\eqref{prop:regularity-solution} and  the symmetry property~\eqref{vectorfield_VV} in Proposition~\ref{proposition}. We have
 \begin{align*}
 & \| \tilde \vartheta_{\nu_q} (t, \cdot) - \vartheta_q (t, \cdot ) \|_{L^2(\T^2)}^2 \\
  \leq & 2 \nu_q \left | \int_0^t \int_{\T^2} \nabla \tilde \vartheta_{\nu_q} (s,x) \cdot \nabla \vartheta_{q}(s,x) dx ds \right | 
\\
  \leq & 2  \left (\nu_q \int_0^t \int_{\T^2} | \nabla \vartheta_{q} (s,x)|^2  dx ds \right )^{1/2} 
 \\
  = & 2 \left (\nu_q \sum_{j = q}^\infty \int_{\mathcal{I}_j\cup \mathcal{J}_j} \int_{\T^2} | \nabla \vartheta_{q} (s,x)|^2  dx ds 
 +
 \nu_q \sum_{j = 0}^{q-1}\int_{\mathcal{I}_j\cup \mathcal{J}_j} \int_{\T^2} | \nabla \vartheta_{q} (s,x)|^2  dx ds 
 \right )^{1/2}
 \\
\le & 
   C \left(  
   a_q^{2 + 3\epsilon} a_q^\gamma a_q^{-2(1+3\epsilon(1+\delta))}
   +
   a_q^{2 + 3\epsilon}
   \sum_{j=0}^{q-1} a_j^{\gamma} a_{j+1}^{-2(1+3\epsilon(1+\delta))}  
   \right)^{1/2}   
 \\ 
 \le & C \left(  
   a_q^{2 + 3\epsilon} a_{q-1}^\gamma a_q^{-2(1+3\epsilon(1+\delta))}  
   \right)^{1/2}
 \\
  \leq & Cq  a_q^{1+ \frac{3\epsilon}{2}} a_{q-1}^\gamma a_q^{-1 - 6 \epsilon }
 \leq Cq a_q^{\frac{\delta}{32} - \frac{9 \epsilon}{2}}  \to 0
 \end{align*}
 as $q \to \infty$,
 where we used $\gamma = \sfrac{\delta}{8}$, $a_{q+1} = a_q^{1+ \delta}$, $\delta \in (0,\sfrac{1}{8})$, $q a_q^{\sfrac{\delta}{32}} \leq 1$, $\sum_{j \geq q} a_j^\gamma \leq 2 a_{q}^\gamma$, and $\sfrac{\delta}{8} > \sfrac{9 \epsilon}{2}$. Therefore, $v_0^{\cc}$ satisfies \eqref{eq:conservative}.

\bibliographystyle{alpha}
 \bibliography{biblio}

\begin{bibdiv}
\begin{biblist}

\bib{ABC}{article}{
      author={Albritton, Dallas},
      author={Bru\`{e}, Elia},
      author={Colombo, Maria},
       title={Non-uniqueness of {L}eray solutions of the forced
  {N}avier-{S}tokes equations},
        date={2022},
        ISSN={0003-486X},
     journal={Ann. of Math. (2)},
      volume={196},
      number={1},
       pages={415\ndash 455},
         url={https://doi.org/10.4007/annals.2022.196.1.3},
      review={\MR{4429263}},
}

\bib{ABCDGJK}{misc}{
      author={Albritton, Dallas},
      author={Bru\`e, Elia},
      author={Colombo, Maria},
      author={De~Lellis, Camillo},
      author={Giri, Vikram},
      author={Janisch, Maximilian},
      author={Kwon, Hyunju},
       title={Instability and nonuniqueness for the $2d$ euler equations in
  vorticity form, after m. vishik},
   publisher={arXiv},
        date={2021},
         url={https://arxiv.org/abs/2112.04943},
}

\bib{BDL22}{misc}{
      author={Bru\`e, Elia},
      author={De~Lellis, Camillo},
       title={Anomalous dissipation for the forced 3d {N}avier-{S}tokes
  equations},
   publisher={arXiv},
        date={2022},
         url={https://arxiv.org/abs/2207.06301},
}

\bib{BDSV}{article}{
      author={Buckmaster, T.},
      author={De~Lellis, C.},
      author={Sz\'ekelyhidi~Jr, L.},
      author={Vicol, V.},
       title={Onsager's conjecture for admissible weak solutions},
        date={2017},
     journal={Comm. Pure Appl. Math., to appear. arXiv:1701.08678},
}

\bib{CCS22}{misc}{
      author={Colombo, Maria},
      author={Crippa, Gianluca},
      author={Sorella, Massimo},
       title={Anomalous dissipation and lack of selection in the
  {O}bukhov-{C}orrsin theory of scalar turbulence},
   publisher={arXiv},
        date={2022},
         url={https://arxiv.org/abs/2207.06833},
}

\bib{ConstantinETiti}{article}{
      author={Constantin, Peter},
      author={E, Weinan},
      author={Titi, Edriss~S.},
       title={Onsager's conjecture on the energy conservation for solutions of
  {E}uler's equation},
        date={1994},
        ISSN={0010-3616},
     journal={Comm. Math. Phys.},
      volume={165},
      number={1},
       pages={207\ndash 209},
         url={http://projecteuclid.org/euclid.cmp/1104271041},
      review={\MR{1298949}},
}

\bib{DLS1}{article}{
      author={De~Lellis, C.},
      author={Sz\'ekelyhidi, L., Jr.},
       title={The {E}uler equations as a differential inclusion},
        date={2009},
        ISSN={0003-486X},
     journal={Ann. of Math. (2)},
      volume={170},
      number={3},
       pages={1417\ndash 1436},
         url={http://dx.doi.org/10.4007/annals.2009.170.1417},
      review={\MR{2600877}},
}

\bib{DLS2}{article}{
      author={De~Lellis, C.},
      author={Sz{{\'e}}kelyhidi, L., Jr.},
       title={Dissipative continuous {E}uler flows},
        date={2013},
        ISSN={0020-9910},
     journal={Invent. Math.},
      volume={193},
      number={2},
       pages={377\ndash 407},
         url={http://dx.doi.org/10.1007/s00222-012-0429-9},
      review={\MR{3090182}},
}

\bib{DRI22}{misc}{
      author={De~Rosa, Luigi},
      author={Isett, Philip},
       title={In preparation.},
}

\bib{DRIS22}{misc}{
      author={De~Rosa, Luigi},
      author={Inversi, Marco},
      author={Stefani, Giorgio},
       title={Weak-strong uniqueness and vanishing viscosity for incompressible
  {E}uler equations in exponential spaces},
   publisher={arXiv},
        date={2022},
         url={https://arxiv.org/abs/2204.12779},
}

\bib{eyinkonsag}{article}{
      author={Eyink, Gregory~L.},
       title={Energy dissipation without viscosity in ideal hydrodynamics. {I}.
  {F}ourier analysis and local energy transfer},
        date={1994},
        ISSN={0167-2789},
     journal={Phys. D},
      volume={78},
      number={3-4},
       pages={222\ndash 240},
         url={https://doi.org/10.1016/0167-2789(94)90117-1},
      review={\MR{1302409}},
}

\bib{Isett}{article}{
      author={Isett, P.},
       title={A proof of {O}nsager's conjecture},
        date={2018},
        ISSN={0003-486X},
     journal={Ann. of Math. (2)},
      volume={188},
      number={3},
       pages={871\ndash 963},
         url={https://mathscinet.ams.org/mathscinet-getitem?mr=3866888},
      review={\MR{3866888}},
}

\bib{JY20}{misc}{
      author={Jeong, In-Jee},
      author={Yoneda, Tsuyoshi},
       title={Quasi-streamwise vortices and enhanced dissipation for the
  incompressible 3d {N}avier-{S}tokes equations},
   publisher={arXiv},
        date={2020},
         url={https://arxiv.org/abs/2012.14621},
}

\bib{JYo20}{misc}{
      author={Jeong, In-Jee},
      author={Yoneda, Tsuyoshi},
       title={Vortex stretching and anomalous dissipation for the
  incompressible 3d {N}avier-{S}tokes equations},
   publisher={arXiv},
        date={2020},
         url={https://arxiv.org/abs/2001.02333},
}

\bib{onsager}{article}{
      author={Onsager, L.},
       title={Statistical hydrodynamics},
        date={1949},
        ISSN={0029-6341},
     journal={Nuovo Cimento (9)},
      volume={6},
      number={Supplemento, 2 (Convegno Internazionale di Meccanica
  Statistica)},
       pages={279\ndash 287},
      review={\MR{36116}},
}

\bib{VishikI}{misc}{
      author={Vishik, Misha},
       title={Instability and non-uniqueness in the cauchy problem for the
  euler equations of an ideal incompressible fluid. part i},
   publisher={arXiv},
        date={2018},
         url={https://arxiv.org/abs/1805.09426},
}

\bib{VishikII}{misc}{
      author={Vishik, Misha},
       title={Instability and non-uniqueness in the cauchy problem for the
  euler equations of an ideal incompressible fluid. part ii},
   publisher={arXiv},
        date={2018},
         url={https://arxiv.org/abs/1805.09440},
}

\bib{W18}{incollection}{
      author={Wiedemann, Emil},
       title={Weak-strong uniqueness in fluid dynamics},
        date={2018},
   booktitle={Partial differential equations in fluid mechanics},
      series={London Math. Soc. Lecture Note Ser.},
      volume={452},
   publisher={Cambridge Univ. Press, Cambridge},
       pages={289\ndash 326},
      review={\MR{3838055}},
}

\end{biblist}
\end{bibdiv}
 
\end{document}